\documentclass[12pt]{amsart}

\usepackage[dvips]{graphicx}
	\usepackage{amsmath,amssymb,amsthm,wrapfig,amsfonts,enumerate,latexsym}

	\newtheorem{dfn}{Definition}[section]
	\newtheorem{thm}[dfn]{Theorem}
	\newtheorem{prop}[dfn]{Proposition}
	\newtheorem{cor}[dfn]{Corollary}
	\newtheorem{lem}[dfn]{Lemma}
	\newtheorem{rem}[dfn]{Remark}
	\newtheorem{ex}[dfn]{Example}
 	\newtheorem{claim}[dfn]{Claim}
	
    \newtheorem{quest}[dfn]{Question}

    \newtheorem{step}{Step}

	\newcounter{yon}
	\numberwithin{equation}{section}

	\newcommand{\dist}{\mathop{\mathit{d}} \nolimits}
	
	\newcommand{\diam}{\mathop{\mathrm{diam}} \nolimits}

	\newcommand{\sep}{\mathop{\mathrm{Sep}} \nolimits}

    \newcommand{\vol}{\mathop{\mathit{vol}}        \nolimits}
    \newcommand{\inrad}{\mathop{\mathrm{inrad}}            \nolimits}

\begin{document}

	\title[Applications of the `Ham Sandwich theorem' to eigenvalues of the Laplacian]
	{Applications of the `Ham Sandwich Theorem' to eigenvalues of the Laplacian}
	\author[Kei Funano]{Kei Funano}
	\email{yahoonitaikou@gmail.com}
	\date{\today}

	\maketitle


	\setlength{\baselineskip}{5mm}

	\begin{abstract}We apply Gromov's ham sandwich method to get
     \begin{enumerate}
     \item domain monotonicity (up to a multiplicative constant factor);
     \item reverse
     domain monotonicity (up to a multiplicative constant
           factor); and
     \item universal
     inequalities
      \end{enumerate}for Neumann eigenvalues of the Laplacian on bounded
     convex domains in a Euclidean space.
	\end{abstract}

\section{Introduction and the statement of main results}

Let $\Omega $ be a bounded domain in $\mathbb{R}^n$ with piecewise
 smooth boundary. We denote by $\lambda_0^D(\Omega)\leq \lambda_1^D(\Omega)\leq \cdots \leq
 \lambda_k^D(\Omega)\leq \cdots $ the Dirichlet eigenvalues of the Laplacian
 on $\Omega$ and by $0=\lambda_0^N(\Omega)<\lambda^N_1(\Omega)\leq \lambda^N_2(\Omega)\leq \cdots
 \leq \lambda^N_k(\Omega)\leq \cdots$ the Neumann eigenvalues of
 the Laplacian on $\Omega$. It is known the following two
properties for these eigenvalues:
\begin{enumerate}
 \item (Domain monotonicity for Dirichlet eigenvalues) If
       $\Omega\subseteq \Omega'$ are two bounded domains, then
       $\lambda_k^D(\Omega')\leq \lambda_k^D(\Omega)$ for any $k$.
 \item (Restricted reverse domain monotonicity for Neumann eigenvalues)
       If in addition
       $\Omega'\setminus \Omega$ has measure zero then
       $\lambda_k^N(\Omega)\leq \lambda_k^N(\Omega')$ for any $k$.
 \end{enumerate}These two properties are a direct consequence of
       Courant's minimax principle (see \cite{chavel}). The following
       two examples suggest that the domain monotonicity does not hold for Neumann eigenvalues in general.
\begin{ex}\label{ex:intro}\upshape Let $\Omega'$ be the $n$-dimensional unit cube
 $[\,0,1\, ]^n$. Then $\lambda_1^N(\Omega')=1$. However if $\Omega$
 is a convex domain in $[\,0,1\,]^n$ that approximates the segment connecting the origin and the
 point $(1,1,\cdots,1)$ then $\lambda_1^N(\Omega)\sim 1/n$. 
 \end{ex}

 \begin{ex}\label{ex:intro2}\upshape Let $p\in [\,1,2\,]$ and $B_p^n$
  be the $n$-dimensional $\ell_p$-ball centered at the origin. Suppose that
  $r_{n,p}$ is the positive number such that $\vol(r_{n,p}B_p^{n})=1$
  and set $\Omega':=r_{n,p}B_p^n$. Then $r_{n,p}\sim n^{1/p}$ and
  $\lambda_1^N(\Omega')\geq c$ for some absolute constant $c>0$
  (\cite[Section 4 (2)]{sodin}). If the segment in $\Omega'$ connecting
  the origin and $(r_{n,p},0,0,\cdots,0)$ is approximated by a convex
  domain $\Omega$ in $\Omega'$ then $\lambda_1^N(\Omega)\sim r_{n,p}^{-2}\sim
  n^{-2/p}$.
  \end{ex}

 In this paper we study the above two properties for Neumann eigenvalues
 of the Laplacian on convex domains in a Euclidean space. For two real
 numbers $\alpha,\beta$ we denote $\alpha\lesssim \beta$ if $\alpha\leq
 c \beta$ for some absolute constant $c>0$.

One of our main theorems is the following:
\begin{thm}\label{Mthm1}For any natural number $k\geq 2$ and any two bounded convex domains
 $\Omega,\Omega'$ in $\mathbb{R}^n$ with piecewise smooth boundaries such that $\Omega\subseteq \Omega'$
 we have
 \begin{align*}
  \lambda^N_k(\Omega')\lesssim  (n \log k)^2 \lambda^N_{k-1}(\Omega). 
  \end{align*}
 \end{thm}

  As a corollary we get the following inner
  radius estimate:
 \begin{cor}Let $\Omega\subseteq \mathbb{R}^n$ be a bounded convex
  domain with piecewise smooth boundary. For any $k\geq 2$ we have 
  \begin{align*}
   \inrad (\Omega)\lesssim \frac{ n \log k \sqrt{\lambda^N_{k-1}(B_1)}}{\sqrt{\lambda^N_k(\Omega)}},
   \end{align*}where $B_1$ is a unit ball in $\mathbb{R}^n$.
  \end{cor}

   We also obtain the opposite inequality to Theorem \ref{Mthm1}:
 \begin{thm}\label{Mthm2}Let $\Omega,\Omega'$ be two bounded convex domains in
  $\mathbb{R}^n$ having piecewise smooth boundaries. Assume that $\Omega$ is
  symmetric with respect to the origin (i.e., $\Omega=-\Omega$) and $\Omega\subseteq \Omega'$. Set $v:= \vol \Omega / \vol \Omega'\in [\, 0,1\,]$. Then for any natural number $k\geq 3$ we have
  \begin{align}\label{Mthm2:s1}
   \lambda^N_{k-2}(\Omega')\gtrsim  \min \Big\{\frac{(\log (1-v))^2}{n^8
   (\log k)^6}, \frac{1}{n^6(\log k)^4}\Big\}\lambda^N_k(\Omega).
   \end{align}For general (not necessarily symmetric) $\Omega$ we have
  \begin{align}\label{Mthm2:s2}
      \lambda^N_{k-2}(\Omega')\gtrsim  \min \Big\{\frac{(\log (1-2^{-n}v))^2}{n^8
   (\log k)^6}, \frac{1}{n^6(\log k)^4}\Big\}\lambda^N_k(\Omega).
   \end{align}
  \end{thm}

    As a corollary of Theorems \ref{Mthm1} and \ref{Mthm2} we obtain
  \begin{align*}
     \lambda^N_k(\Omega')\lesssim  (n \log k)^2 \lambda^N_{k}(\Omega) \text{ and }\lambda^N_{k}(\Omega')\gtrsim \min \Big\{\frac{(\log (1-2^{-n}v))^2}{n^8
   (\log k)^6}, \frac{1}{n^6(\log k)^4}\Big\}\lambda^N_k(\Omega)
   \end{align*}for all $k\geq 2$, which corresponds to the above
   properties $(1)$ and $(2)$ up to multiplicative constant factors. In \cite{milman1} E.~Milman obtained the corresponding
   inequality for $k=1$ (see (\ref{eneq:emil})). Despite the fact that his inequality is independent of
   dimension, our two inequalities above involve dimensional
   terms. However $\log k$ bounds in the two inequalities are
   nontrivial (Compare with (\ref{eneq:last})). The case where $p=1$ in Example \ref{ex:intro2}
   shows that the $n^2$ order in Theorem
   \ref{Mthm1} cannot be improved. Probably there would be a chance to express
   the multiplicative constant factor in Theorem \ref{Mthm1} in terms of the volume ratio $v=\vol
   \Omega/\vol \Omega'$ to avoid the dependence of dimension (see Question \ref{quest3}).

  As a special case where $\Omega=\Omega'$ in Theorem \ref{Mthm1} we
  obtain the following universal inequalities
  among Neumann eigenvalues :
  \begin{align}\label{Mthm:ineq}
   \lambda^N_k(\Omega)\lesssim (n\log k)^2 \lambda^N_{k-1}(\Omega).
   \end{align}By `universal' we mean it does not depend on the
   underlying domain $\Omega$ itself. Payne, P\'olya, and
Weinberger studied universal inequalities among Dirichlet eigenvalues (\cite{ppw1,ppw2}). Since then many universal
inequalities for Dirichlet eigenvalues were studied (see
\cite{ashben}). For Neumann eigenvalues, Liu (\cite{liu}) showed the sharp
inequalities
\begin{align}\label{lineq}
 \lambda^N_k(\Omega)\lesssim k^2\lambda^N_1(\Omega)
 \end{align}for any bounded convex domain $\Omega$, which improves author's exponential bounds in $k$ in \cite{funa-first}. On the other
 hand, one can get
 \begin{align}\label{emineq}
  \lambda^N_k(\Omega) \gtrsim k^{2/n}\lambda^N_1(\Omega)
 \end{align}for any bounded convex domain $\Omega\subseteq \mathbb{R}^n$. This
 inequality follows from the combination of E.~Milman's result
 \cite[Remark 2.11]{milman1} and Cheng-Li's result \cite{cl81} (see
 \cite[Chapter III \S 5]{sy94}). In fact E.~Milman described the Sobolev
 inequality in terms of $\lambda^N_1(\Omega)$ and Cheng-Li showed lower bounds
 of $\lambda^N_k(\Omega)$ in terms of the Sobolev
 constant. The Weyl asymptotic formula says that the inequality
 (\ref{emineq}) is sharp. In particular combining (\ref{lineq}) with (\ref{emineq}) we
 can obtain
 \begin{align*}
 \lambda^N_k(\Omega)\lesssim
 k^{2-2/n}\lambda^N_{k-1}(\Omega).
  \end{align*}Comparing with this inequality our inequality (\ref{Mthm:ineq}) includes the dimensional term. However the dependence
 on $k$ is best ever to author's knowledge. It should be mentioned
 that author's conjecture in \cite{funa-first,funafuna} is
 $\lambda^N_k(\Omega)\lesssim \lambda^N_{k-1}(\Omega)$ for any bounded
 convex domain $\Omega$ with piecewise smooth boundary.


In the proof of Theorems \ref{Mthm1} and \ref{Mthm2} we will use
Gromov's method concerning a bisection of finite subsets by the zero set
of a finite combination of eigenfunctions. It enables us to get lower bounds for
eigenvalues of the Laplacian in terms of Cheeger constants and the maximal
multiplicity of a covering of a domain (Proposition
\ref{bisect:prop}). We will try to find `nice' convex partition in order
to get `nice' lower bounds for Cheeger constants of pieces of the partition. 

 \section{Preliminaries}
 \subsection{Separation distance}
 Let $\Omega$ be a bounded domain in a Euclidean space. For two subsets
 $A,B\subseteq \Omega$ we set $\dist_{\Omega}(A,B):=\inf \{ |x-y| \mid
 x\in A, y\in B\}$. We denote by $\mu$ the Lebesgue measure on $\Omega$ normalized as
$\mu(\Omega)=1$.
  \begin{dfn}[Separation distance, \cite{gromov}]\label{2.2d1}\upshape 
  For any $\kappa_0,\kappa_1,\cdots,\kappa_k\geq 0$ with $k\geq 1$,
  we define the ($k$-)\emph{separation distance}
  $\sep(\Omega;\kappa_0,\kappa_1, \cdots, \kappa_k)$ of $\Omega$
  as the supremum of $\min_{i\neq j}\dist_{\Omega}(A_i,A_j)$,
  where $A_0,A_1, \cdots, A_k$ are
  any Borel subsets of $\Omega$ satisfying that $\mu(A_i)\geq \kappa_i$ for
  all $i=0,1,\cdots,k$.
\end{dfn}

  \begin{thm}[{\cite[Theorem 1]{funafuna}}]\label{reduction:thm}There
   exists an absolute constant $c>0$
 satisfying the following property. Let $\Omega$ be a bounded convex
   domain in a Euclidean space with piecewise smooth boundary and $k,l$ be two natural numbers with
   $l\leq k$. Then we have 
   \begin{align*}
   \sep(\Omega;\kappa_0,\cdots,\kappa_{l})\leq
   \frac{c^{k-l+1}}{\sqrt{\lambda^N_k(\Omega)}}\max_{i\neq j} \log
    \frac{1}{\kappa_i\kappa_j}.
    \end{align*}
   \end{thm}

   The case where $k=l=1$ was first proved by Gromov and V.~Milman
   without convexity assumption of domains (\cite{milgro}). Chung, Grigor'yan, and Yau
   then extended to the case where $k=l$ (\cite{cgy1,cgy2}). To reduce the number $l$ of
   subsets in $\Omega$ in a dimension-free way we need the convexity of
   $\Omega$ (see \cite{funafuna}).

\subsection{Cheeger constant and eigenvalues of the Laplacian}
    For a Borel subset $A\subseteq \Omega$ and $r>0$ we
denote $U_r(A)$ the $r$-neighborhood of $A$ in $\Omega$. We define the \emph{Minkowski
boundary measure} of $A$ as
\begin{align*}
 \mu_{+}(A):=\liminf_{r\to 0} \frac{\mu(U_r(A) \setminus A)}{r}. 
 \end{align*}
  \begin{dfn}[{Cheeger constant}]\upshape For a bounded domain $\Omega$ in a
   Euclidean space we define the \emph{Cheeger constant} of $\Omega$ as
  \begin{align*}
   h(\Omega):= \inf_{A_0,A_1} \max \{ 
   \mu_{+}( A_0)/\mu(A_0),\mu_{+}(A_1)/\mu(A_1) \},
   \end{align*}where the infimum runs over all non-empty disjoint two Borel subsets $A_0,A_1$ of $\Omega$. 
  \end{dfn}
Let $\mu$ be a finite Borel measure on a bounded domain
$\Omega\subseteq \mathbb{R}^n$ and $f:\Omega\to \mathbb{R}$ be a Borel
measurable function. A real number $m_f$ is called a \emph{median} of
$f$ if it satisfies 
\begin{align*}
 \mu(\{x\in \Omega\mid f(x)\geq m_f\})\geq \mu(\Omega)/2 \text{ and } \mu(\{x\in \Omega\mid f(x)\leq m_f\})\geq \mu(\Omega)/2.
 \end{align*}
 The following characterization of the Cheeger constant is due to Maz'ya
 and Federer-Fleming. See \cite[Lemma 2.2]{milman1} for example.
  \begin{thm}[{\cite{ff60}, \cite{mazya}}]\label{11poin}The Cheeger constant $h(\Omega)$ is the best constant for the following
   $(1,1)$-Poincar\'e inequality:
   \begin{align*}
    h(\Omega) \| f-m_f\|_{L^1(\Omega,\mu)}\leq \| |\nabla f|\|_{L^1(\Omega,\mu)}
    \text{ for any }f\in C^{\infty}(\Omega). 
    \end{align*}
   \end{thm}

  \begin{thm}[{E.~Milman \cite[Theorem 2.1]{milman2}}]\label{emil:thm}Let $\Omega $ be a
   bounded convex domain in a Euclidean space and assume that $\Omega$ satisfies
   the following concentration inequality for some $r>0$ and $\kappa\in
   (\, 0,1/2\,)$ : $\mu(\Omega \setminus U_{r}(A))\leq \kappa$ for any
   Borel subset $A\subseteq \Omega$ such that $\mu(A)\geq 1/2$. Then $h(\Omega)\geq (1-2\kappa)/r$.
   \end{thm}One can easily check that
   Theorem \ref{emil:thm} has the following equivalent interpretation in terms
   of separation distance.
   \begin{prop}\label{emil:prop}Let $\Omega$ be a convex domain in a Euclidean
    space. Then for any $\kappa\in (\,0,1/2\,)$ we have
    \begin{align*}
     \sep(\Omega;\kappa,1/2)\geq (1-2\kappa)/h(\Omega).
     \end{align*}In particular we have
    \begin{align*}
    \diam \Omega \geq 1/h(\Omega).
     \end{align*}
    \end{prop}
    The latter statement can
    be found in \cite[Theorem 5.1]{kls} and \cite[Theorem 5.12]{milman1} up to some
    absolute constant.
    \begin{thm}[{\cite[Theorem 1.1]{kr99}, \cite{cheng}}]\label{diam est1}Let $\Omega$ be a bounded convex domain in
     $\mathbb{R}^n$ with piecewise smooth boundary. For any natural number $k$
     we have
     \begin{align*}
      \diam \Omega \lesssim n k/\sqrt{\lambda^N_k(\Omega)}.
      \end{align*}
     \end{thm}
     The Buser-Ledoux inequality asserts that $\sqrt{\lambda^N_1(\Omega)}\gtrsim
     h(\Omega)$ for any bounded convex domain $\Omega\subseteq \mathbb{R}^n$
     with piecewise smooth boundary (\cite{buser}, \cite{led}). As a corollary of
     Theorem \ref{diam est1} we obtain
     \begin{align}\label{diam est}
      \diam \Omega \lesssim n/h(\Omega).
      \end{align}
  \subsection{Voronoi partition}
  Let $X$ be a metric space and $\{x_i\}_{i\in I}$ be a subset of
  $X$. For each $i\in I$ we define the \emph{Voronoi cell} $C_i$
  associated with the point $x_i$ as
  \begin{align*}
   C_i:= \{x\in X \mid \dist(x,x_i)\leq \dist(x,x_j) \text{ for all
   }j\neq i   \}.
   \end{align*}Note that if $X$ is a bounded convex domain $\Omega$ in a
   Euclidean space then $\{C_i\}_{i\in I}$ is a convex partition of
   $\Omega$ (the boundaries $\partial C_i$ may overlap each
   other). Observe also that if the balls $\{ B(x_i,r)\}_{i\in I}$ of
   radius $r$ covers $\Omega$ then $C_i \subseteq B(x_i,r)$, and thus
   $\diam C_i \leq 2r$ for any $i\in I$. 
    \section{Gromov's ham sandwich method}
    In this section we explain Gromov's ham sandwich method to estimate 
    eigenvalues of the Laplacian from below. Recall that the classical ham
    sandwich theorem in algebraic topology asserts that given three
    finite volume subsets in $\mathbb{R}^3$, there is a plane that
    bisects all these subsets (\cite{matousek}). In stead of bisecting by a plane we
    consider bisecting by the zero set of a finite combination of
    eigenfunctions of the Laplacian in Gromov's ham sandwich method.

    Let $\Omega$ be a bounded domain in a Euclidean space with piecewise smooth
    boundary and $\{A_i\}_{i=1}^{l}$ be a finite covering of $\Omega$;
    $\Omega=\bigcup_i A_i$.
    We denote by $M(\{A_i\})$ the maximal multiplicity of the covering
    $\{ A_i\}$ and by $h(\{ A_i \})$ the minimum of the Cheeger constants of
    $A_i$, $i=1,2,\cdots,l$.

    Although the following argument is
    essentially included in
    \cite[Appendix $C_+$]{gromov} we include the proof for the
    completeness of this paper.
    \begin{prop}[{Compare with \cite[Appendix $C_{+}$]{gromov}}]\label{bisect:prop}Under
     the above situation, we have
     \begin{align*}
      \lambda^N_l(\Omega)\geq \frac{h(\{A_i\})^2}{4M(\{A_i\})^2}.
      \end{align*}
     \begin{proof}[{Sketch of Proof}]We abbreviate $M:=M(\{A_i\})$ and
      $h:=h(\{A_i\})$. Take orthonormal eigenfunctions $f_1,f_2,\cdots,f_l$ which
      correspond to the eigenvalues $\lambda^N_1(\Omega),\lambda^N_2(\Omega),\cdots,\lambda^N_l(\Omega)$ respectively.
  \begin{step}\label{step1}\upshape Use the Borsuk-Ulam theorem to get constants
  $c_0,c_1,\cdots,c_{l}$ such that $f:=c_0 + \sum_{i=1}^{l}c_i f_i$
  bisects each $A_1, A_2,\cdots A_{l}$, i.e.,
   \begin{align*}
    \mu(A_i\cap
   f^{-1}[\, 0,\infty\,))\geq \mu(A_i)/2 \text{ and } \mu(A_i\cap
    f^{-1}(\,-\infty,0\,])\geq \mu(A_i)/2.
    \end{align*}In fact, according to \cite[Corollary]{stonetukey}, in
   order to
   bisect $l$ subsets by a finite combination of $f_0\equiv 1,
   f_1,\cdots,f_l$, it suffices to check that $f_0,f_1,\cdots, f_l$ are
   linearly independent modulo sets of measure zero (i.e., whenever
   $a_0f_0+a_1f_1+\cdots +a_lf_l=0$ over a Borel subset of positive
   measure, we have $a_0=a_1=\cdots =a_l=0$). This
   is possible since the zero set of any finite combination of
   $f_0,f_1,\cdots,f_l$ has finite codimension $1$ Hausdorff measure
   (\cite[Subsection 1.1.1]{bhh}).
   \end{step}
  \begin{step}\label{step2}\upshape Put $f_+(x):=\max\{ f(x),0\}$ and
   $f_{-}(x):=\max \{ -f(x),0\}$. Then we set $g_{\pm}:=f_{\pm}^2$. Note that $0$ is
      the median of the restriction of $g_{\pm}$ to each $A_i$ by Step
   \ref{step1}. Apply Theorem \ref{11poin} to get $h \| g_{\pm}\|_{L^1(A_i,\mu|_{A_i})}\leq
   \|| \nabla g_{\pm}|\|_{L^1(A_i,\mu|_{A_i})}$ for each $i$.
   \end{step}
   \begin{step}\upshape
    Use Step \ref{step2} to get
    \begin{align*}
     \int_{\Omega} g_{\pm} d\mu \leq \sum_{i=1}^{l}\int_{A_i} g_{\pm} d\mu  \leq
     \frac{1}{h}\sum \int_{A_i} | \nabla g_{\pm} | d\mu \leq
     \frac{M}{h}\int_{\Omega} | \nabla g_{\pm} | d\mu.
     \end{align*}
   \end{step}\hspace{-0.5cm}Recalling that $g_{\pm}=f_{\pm}^2$ and using the Cauchy-Schwarz
    inequality we have
    \begin{align*}
     \int_{\Omega} f_{\pm}^2 d\mu \leq \frac{4M^2}{h^2}\int_{\Omega} | \nabla
     f_{\pm} |^2 d\mu. 
     \end{align*}Since the zero set $f^{-1}(0)$ has measure zero we
      get
      \begin{align*}
         \int_{\Omega} f^2 d\mu \leq \frac{4M^2}{h^2}\int_{\Omega} | \nabla
     f |^2 d\mu. 
       \end{align*}We therefore obtain $\sum_{i=0}^{l}c_i^2 \leq
    (4M^2/h^2)\sum_{i=1}^l c_i^2\lambda^N_i(\Omega)$ and thus the
    conclusion of the proposition.
      \end{proof}
     \end{prop}
     \begin{rem}\upshape 1. In \cite{gromov} Gromov treated the case where
      $\Omega$ is a closed Riemannian manifold of Ricci curvature $\geq
      -(n-1)$ and the covering consists of some balls $B_i$ of radius $\varepsilon$ in $\Omega$. In stead of considering the $(1,1)$-Poincar\'e
      inequality in terms of Cheeger constants in Step \ref{step2} he proved that
      $\|g\|_{L^1(B_i,\mu|_{B_i})} \leq c(n,\varepsilon)\| |\nabla
      g|\|_{L^1(\widetilde{B}_i,\mu|_{\widetilde{B}_i})}$, where
      $g=f^2$, $c(n,\varepsilon)$ is a constant depending only on dimension $n$
      and $\varepsilon$, and $\widetilde{B}_i$ is the ball of radius
      $2\varepsilon$ with the same center of $B_i$.

      2. The above proposition is also valid for the case where $\Omega$
      is a closed Riemannian manifold or a compact Riemannian manifold with
      boundary. In the latter case we impose the Neumann boundary condition.
      \end{rem}

      As an application of Proposition \ref{bisect:prop} we can obtain
      estimates of eigenvalues of the Laplacian of closed hyperbolic
      manifolds due to Buser (\cite[Theorems 3.1, 3.12, 3.14]{buser2}).
      In fact Buser gave a partition of a closed hyperbolic manifold
      and lower bound estimates of Cheeger constants of each piece of the partition. 
    \section{Proof of main theorems}
    Let $\Omega,\Omega'$ be two bounded convex domains in a Euclidean
    space. Throughout this section $\mu$ is the Lebesgue measure on
    $\Omega'$ normalized as $\mu(\Omega')=1$.
  \begin{proof}[Proof of Theorem
   \ref{Mthm1}]We apply Gromov's ham sandwich method (Proposition \ref{bisect:prop}) to
   bound $\lambda^N_{k-1}(\Omega)$ from below in terms of
   $\lambda^N_k(\Omega')$. To apply the proposition we want to find a
   finite partition $\{ \Omega_i\}_{i=1}^{l}$ of $\Omega$ with $l\leq
   k-1$ such that the Cheeger constant of each $\Omega_i$ can be comparable
   with $\sqrt{\lambda^N_k(\Omega')}$.

   According to Theorem \ref{reduction:thm} we have
   \begin{align}\label{k sep ineq}
    \sep\Big(\Omega';\underbrace{\frac{1}{k^n},\frac{1}{k^n},
    \cdots,\frac{1}{k^n}}_{k\text{ times}}\Big)\leq \frac{c n \log k}{\sqrt{\lambda^N_k(\Omega')}}
   \end{align}for some absolute constant $c>0$. We set $R:=(c n \log k) /\sqrt{\lambda^N_k(\Omega')}$.

   Suppose that $\Omega'$ includes $k$ ($4R$)-separated points
   $x_1,x_2,\cdots,x_{k}$. By Theorem \ref{diam est1} we have $\diam \Omega' \leq
   c'n k/ \sqrt{\lambda^N_k(\Omega')}$ for some absolute constant
   $c'>0$. Applying the Bishop-Gromov inequality we have
   \begin{align*}
    \mu (B(x_i,R))\geq (R/\diam \Omega')^n\geq (c\log k)^n/(c'k)^n
    \end{align*}for each $i$. If we rechoose $c$ in (\ref{k sep ineq})
   as a sufficiently large absolute constant so that $(c \log k)/c'\geq
   1$ we get $\mu(B(x_i,R))\geq 1/k^n$. Since $B(x_i,R)$'s are $2R$-separated this
   contradicts to (\ref{k sep ineq}). 

   Let $y_1,y_2,\cdots,y_l$ be maximal $4R$-separated points in
   $\Omega'$, where $l\leq k-1$. Since $\Omega'\subseteq
   \bigcup_{i=1}^{l} B(y_i,4R)$ if $\{
   \Omega_i' \}_{i=1}^{l}$ is the Voronoi partition associated with
   $\{y_i\}$ then we have $\diam \Omega_i'\leq 8R$. Setting
   $\Omega_i:=\Omega_i'\cap \Omega$ we get
   $\Omega=\bigcup_{i=1}^l\Omega_i$ and $\diam \Omega_i\leq 8R$. Since each
   $\Omega_i$ is convex, Proposition \ref{emil:prop} gives
   $h(\Omega_i)\geq 1/(8R)$. Applying Proposition \ref{bisect:prop} to
   the covering $\{\Omega_i\}$ we obtain
   \begin{align*}
    \lambda^N_{k-1}(\Omega)\geq \lambda^N_l(\Omega)\geq 1/\{4(8 R)^2\}\geq
    \lambda^N_k(\Omega')/(16cn\log k)^2,
    \end{align*}which yields the conclusion of the theorem. This completes the proof.
  \end{proof}

  In order to prove Theorem \ref{Mthm2} we prepare several lemmas.
  \begin{lem}[{\cite[Lemma 5.2]{milman1}}]\label{lem2:mthm2}Let $\Omega,\Omega'$
   be two bounded convex domains in $\mathbb{R}^n$ such that
   $\Omega\subseteq \Omega'$. Assume that $\vol
   \Omega\geq v\vol \Omega'$. Then we have $h(\Omega')\geq v^2h(\Omega)$.
   \end{lem}
  \begin{lem}\label{lem:mthm2}Let $\Omega, \Omega'$ be two bounded convex
   domains in $\mathbb{R}^n$ having piecewise smooth boundaries such that
   $\Omega\subseteq \Omega'$. Assume that $\vol
   \Omega \geq (1-k^{-n})\vol \Omega'$ for some natural number $k\geq 2$. Then we have
   \begin{align*}
    (n^2\log k)^2 \lambda^N_{k-1}(\Omega')\gtrsim \lambda^N_k(\Omega).
    \end{align*}
   \begin{proof}
    Due to Theorem \ref{reduction:thm} we have
   \begin{align}\label{k-1 sep}
    \sep\Big(\Omega;\underbrace{\frac{1}{k^n},\frac{1}{k^n},
    \cdots,\frac{1}{k^n}}_{k\text{ times}}\Big)\leq \frac{c n \log k}{\sqrt{\lambda^N_{k}(\Omega)}}.
    \end{align}We set $R:=(cn^2\log k) / \sqrt{\lambda^N_{k}(\Omega)}$. As in the proof of Theorem \ref{Mthm1} we have maximal
   $4R$-separated points $x_1,x_2,\cdots,x_l\in {\Omega}$ such
   that $l\leq k-1$. We get $\Omega\subseteq \bigcup_{i=1}^l
   B(x_i,4R)$.
   \begin{claim}$U_{R}({\Omega})=\Omega'$.
    \end{claim}Let us admit the above claim for a while. The above claim
    yields $\Omega'\subseteq \bigcup_{i=1}^l B(x_i,5R) $. Let
   $\{\Omega_i'\}_{i=1}^l$ be the Voronoi partition associated with
   $\{x_i\}$ then we have $\diam \Omega_i'\leq 10 R$. Proposition \ref{emil:prop}
   gives $h(\Omega_i')\geq 1/(10R)$. According to Proposition
    \ref{bisect:prop} we obtain
    \begin{align*}
     \lambda^N_{k-1}(\Omega')\geq 1/(20R)^2= \lambda^N_k(\Omega)/(20cn^2 \log k)^2,
     \end{align*}which implies the lemma.

    Suppose that $U_{R}(\Omega)\neq \Omega'$. There exists $x\in
    \partial \Omega'$ such that $B(x,R)\cap \Omega =\emptyset$.
    Lemma \ref{lem2:mthm2} together with Proposition \ref{emil:prop} and (\ref{diam est}) show that
   \begin{align*}
    \diam \Omega' \lesssim n/h(\Omega') \lesssim n/h({\Omega})\leq n
    \diam {\Omega},
    \end{align*}which gives the existence of an absolute constant $c_1>0$
    such that $\diam \Omega' \leq c_1 n\diam \Omega$. The Bishop-Gromov inequality yields
   \begin{align*}
    \mu (B(x,R))\geq (R/\diam \Omega')^n\geq R^n/(c_1 n
    \diam {\Omega})^n.
    \end{align*}Since $\diam {\Omega}\leq c_2
   nk/\sqrt{\lambda^N_{k}({\Omega})}$ for some absolute constant $c_2>0$
    (Theorem \ref{diam est1})
    we have
   \begin{align*}
    \mu(B(x,R))\geq (c\log k)^n/(c_1 c_2 k)^n> 1/k^n,
    \end{align*}provided that $c$ is large enough absolute constant such
   that $(c\log k)/(c_1c_2) > 1$. We thereby obtain
   \begin{align*}\mu(B(x,R)\cup
   {\Omega})=\mu(B(x,R))+\mu({\Omega})> 1/k^n + (1-1/k^n)=1,
    \end{align*}which is a contradiction. 
    \end{proof}
   \end{lem}

   In order to adapt to the hypothesis of Lemma \ref{lem:mthm2} we use the following improvement of Borell's lemma.
   \begin{thm}[{\cite[Section 1 Remark]{guedon}}]\label{guedon:thm}Let $\Omega,\Omega'$ be two
    bounded convex domains such that $\Omega\subseteq \Omega'$. Assume that $\Omega$ is symmetric. Then
    for any $r\geq 1$ we have
    \begin{align*}
     \mu(\Omega'\setminus r\Omega)\leq (1-\mu(\Omega))^{\frac{r+1}{2}},
     \end{align*}where $r\Omega:=\{ rx \mid x\in \Omega\}$.
    \end{thm}


  \begin{proof}[Proof of Theorem \ref{Mthm2}]We first consider the case
   where $\Omega$ is symmetric. According to Theorem \ref{guedon:thm}
   setting 
   \begin{align*}
    r:=2\max \Big\{ \frac{n\log k}{-\log (1-v)}, 1\Big\}
    \end{align*}we have $\mu(\Omega'\setminus r\Omega)<1/k^n$.
   Take a bounded convex domain $\widetilde{\Omega}\subseteq r\Omega
   \cap \Omega'$ with piecewise smooth boundary such that
   $\Omega\subseteq \widetilde{\Omega}$ and $\mu(\Omega'\setminus
   \widetilde{\Omega})<1/k^n$. Since $\widetilde{\Omega}\subseteq
   r\Omega$ Theorem \ref{Mthm1} implies
   \begin{align}\label{eigen comp1}
    (nr\log k)^2\lambda^N_{k-1}(\widetilde{\Omega})\gtrsim r^2\lambda^N_k(r\Omega)=\lambda^N_k(\Omega).
    \end{align}Using Lemma \ref{lem:mthm2} we
   also obtain
   \begin{align}\label{eigen comp2}
    (n^2\log k)^2\lambda^N_{k-2}(\Omega')\gtrsim \lambda^N_{k-1}(\widetilde{\Omega}).
   \end{align}Combining the above two inequalities (\ref{eigen comp1})
   and (\ref{eigen comp2})
   we obtain (\ref{Mthm2:s1}).

   For general (not necessarily symmetric) $\Omega$, there exists a
   choice of a center (we may assume here that the center is the origin
   without loss of generality) such that $\vol (\Omega \cap -\Omega)\geq
   2^{-n}\vol \Omega$ (\cite[Corollary]{stein}). By virtue of Theorem \ref{guedon:thm} setting 
   \begin{align*}
     r:=2\max \Big\{ \frac{n\log k}{-\log (1-2^{-n}v)}, 1\Big\}
    \end{align*}we get $\mu(\Omega'\setminus r\Omega)\leq
   \mu(\Omega'\setminus r(\Omega\cap -\Omega))<1/k^n$. Thus applying the same proof of the
   symmetric case we obtain (\ref{Mthm2:s2}). This completes the proof.
   \end{proof}

   \section{Questions}
   In this section we raise several questions which concern this paper.
 \begin{quest}\label{quest1}Let $\Omega$ be a bounded convex domain in $\mathbb{R}^n$
  with piecewise smooth boundary. Then for any natural number $k$ and
  $\kappa_0,\kappa_1,\cdots,\kappa_k>0$ can we get
  \begin{align*}
   \sep(\Omega;\kappa_0,\kappa_1,\cdots,\kappa_k)\lesssim \frac{1}{(\log k)
   \sqrt{\lambda^N_k(\Omega)}}\max_{i\neq j} \log
   \frac{1}{\kappa_i\kappa_j} \ ?
   \end{align*}
  \end{quest}We can subtract $\log k$ terms in Theorems \ref{Mthm1} and \ref{Mthm2}
   once we get an affirmative answer to Question \ref{quest1} since
   \cite[Theorem 3.4]{funafuna} gives
   \begin{align*}
    \sep(\Omega;\kappa_0,\kappa_1,\cdots,\kappa_l) \leq
    \frac{c^{k-l+1}}{(\log k)\sqrt{\lambda^N_k(\Omega)}} \max_{i\neq
    j}\log \frac{1}{\kappa_1\kappa_j}
    \end{align*}for any two natural numbers $l\leq k$ and any
    $\kappa_0,\kappa_1,\cdots, \kappa_l>0$, where $c>0$ is an absolute constant. 
 \begin{quest}\label{quest2}Let $\Omega$ be a bounded convex domain in $\mathbb{R}^n$ with
  piecewise smooth boundary and assume that $\Omega$ satisfies the following
  $(k-1)$-separation inequality for some $k$:
  \begin{align*}
   \sep((\Omega,\mu);\kappa_0,\kappa_1,\cdots,\kappa_{k-1})\leq
   \frac{1}{D}\max_{i\neq j}\log \frac{1}{\kappa_i\kappa_j} \ \ (\forall \kappa_0,\kappa_1,\cdots,\kappa_{k-1}>0).
   \end{align*}Then do there exist an absolute constant $c>0$ and a convex partition $\Omega=
  \bigcup_{i=1}^{l}\Omega_i$ with $l\leq k-1$ such that
  \begin{align*}
   \mu(\Omega_i)\geq \frac{1}{ck} \text{ and }
   \sep ((\Omega_i, \mu|_{\Omega_i});\kappa,\kappa)\leq
   \frac{c}{D}\log\frac{1}{\kappa} 
   \end{align*}for any $\kappa$ ?
  \end{quest}
An affirmative answer to Question \ref{quest2} implies the universal
inequality $\lambda^N_k(\Omega)\lesssim (\log k)^2\lambda^N_{k-1}(\Omega)$
via Theorem \ref{emil:thm} and Proposition \ref{bisect:prop}. If both
Questions \ref{quest1} and \ref{quest2} is affirmative then we can
obtain $\lambda^N_k(\Omega)\lesssim \lambda^N_{k-1}(\Omega)$.
\begin{quest}\label{quest3} Let $\Omega,\Omega'$ be two bounded convex
 domains with piecewise smooth boundaries such that $\Omega\subseteq
 \Omega'$. Set $v:=\vol \Omega/\vol \Omega'\in [\,0,1\,]$. Can we prove
 $\lambda_k^N(\Omega)\leq f_1(v) g_1(\log k) \lambda_k^N(\Omega')$ and
 $\lambda_k(\Omega') \leq f_2(v) g_2(\log k) \lambda_k^N(\Omega)$, where $f_1$ and $f_2$ are any functions and $g_1$
 and $g_2$ are some rational functions ? 
 \end{quest}
 When $k=1$ E.~Milman obtained
 \begin{align}\label{eneq:emil}
  \lambda_1^N(\Omega')\geq v^4
 \lambda_1^N(\Omega) \text{ and }\lambda_1^N(\Omega)\gtrsim (1/\log(1+1/v))^2 \lambda_1^N(\Omega')
  \end{align}(see \cite[Lemmas 5.1, 5.2]{milman1}). Combining this
  inequality with (\ref{lineq}) and (\ref{emineq}) we can get
  \begin{align}\label{eneq:last}
   \lambda_k^N(\Omega')\gtrsim v^4k^{\frac{2}{n}-2} \lambda_k^N(\Omega) \text{ and
   }\lambda_k^N(\Omega)\gtrsim (k^{\frac{1}{n}-1}/\log(1+1/v))^2\lambda_k^N(\Omega'),
   \end{align}but this does not imply the answer to Question \ref{quest3}.


\bigbreak
\noindent
{\it Acknowledgments.} The author would like to express his appreciation to Prof. Larry
Guth for suggesting Gromov's ham sandwich method and many
stimulating discussion. 
The author also would like to thank to Prof. Emanuel Milman for
useful comments, especially pointing out Examples \ref{ex:intro} and \ref{ex:intro2}, and to Prof. Alexander Grigor'yan for corrections of the
$1$st version of this paper and useful comments.
The author is grateful to Prof. Boris Hanin, Prof. Ryokichi Tanaka, and Dr. Gabriel Pallier
for useful comments and anonymous referees for their helpful suggestions. This work was done during the
author's stay in MIT. The author thanks for its hospitality and
stimulating research environment.

	\end{document}